\documentclass[10pt,A4]{article}
\usepackage[latin5]{inputenc}
\usepackage{amsmath,amssymb}
\usepackage{amsthm}
\usepackage{indentfirst}

\newtheorem{definition}{Definition}
\newtheorem{theorem}{Theorem}

\newtheorem{proposition}{Proposition}

\numberwithin{definition}{section} \numberwithin{theorem}{section}
\numberwithin{lemma}{section}\numberwithin{corollary}{section}
\numberwithin{equation}{section} \numberwithin{example}{section}
\numberwithin{proposition}{section} \numberwithin{remark}{section}
\begin{document}

\begin{center}

{\bf \Large Approximation of the Set of Integrable Trajectories of the Control Systems with Limited Control Resources}

\vspace{3mm}

Nesir Huseyin$^1$, Anar Huseyin$^2$, Khalik G. Guseinov$^3$

\vspace{3mm}

$^1$Cumhuriyet University, Faculty of Education, Department of Mathematics and Science Education \\ 58140 Sivas, TURKEY

e-mail: nhuseyin@cumhuriyet.edu.tr

\vspace{2mm}

{\small $^2$Cumhuriyet University, Faculty of Science, Department of Statistics and Computer Sciences \\ 58140 Sivas, TURKEY

e-mail: ahuseyin@cumhuriyet.edu.tr

\vspace{2mm}

$^3$Eskisehir Technical  University, Faculty of Science, Department of Mathematics \\ 26470 Eskisehir, TURKEY

e-mail: kguseynov@eskisehir.edu.tr}

\end{center}

\textbf{Abstract.} In this paper an approximation of the set of multivariable and $L_2$ integrable trajectories of the control system described by Urysohn type integral equation is considered. It is assumed that the system is affine with respect to the control vector. The admissible control functions are chosen from the closed ball of the space $L_2$, centered at the origin with radius $\rho$. Step by step way, the set of admissible control functions is replaced by the set of controls, consisting of a finite number of piecewise-constant control functions. It is proved that under appropriate choosing of the discretization parameters, the set of trajectories generated by a finite number of piecewise-constant control functions is an internal approximation of the set of trajectories.

\vspace{5mm}


\vspace{5mm}


\section{Introduction}

One of the important notions of the theory of control systems are the attainable set and the integral funnel concepts of a given system. In the case when the trajectories of the system are continuous functions, the attainable set and the integral funnel of the control system have rather simple geometric interpretations. The attainable set is defined in the space of states  and consists of points to which the trajectories of the system arrive at the given instant of time. The integral funnel is considered as a generalization of the integral curve notion from the theory of ordinary differential equations and is defined in the space of positions which consists of the graphs of all possible trajectories of the system. Various topological properties and methods of approximate construction of the attainable sets and the integral funnel have been studied in a vast number of papers (see, e.g. \cite{buz} -- \cite{pat} and references therein). It should also be noted the studies that were carried out in the framework of the theory of differential inclusions (see, e.g. \cite{aub} -- \cite{pan}). The construction of the attainable sets and integral funnel of the given control system, allows to construct the trajectories with different prescribed properties (see, e.g. \cite{ers}).

When the trajectories of the system are integrable functions, then in this case the concepts of the attainable set and integral funnel lose their original geometric interpretations. In this case, an important tool for investigation of the control system is the notion of the set of trajectories, which consists of integrable trajectories generated by all possible admissible control functions.

Depending on the character of the control functions, the control systems are classified as the systems with geometric constraints on the control functions; the systems with integral constraints on the control functions; and the systems with mixed, i.e. with geometric and integral constraints on the control functions.  An integral constraint on the control functions is inevitable if a control resource is exhausted by consumption, such as energy, fuel, finance, food, etc. (see, e.g. \cite{con} -- \cite{sub}). Investigation of the attainable sets and integral funnels of the control systems described by ordinary differential equations with integral constraints on the control functions are discussed in papers  \cite{gus3} -- \cite{rou} (see also references therein).

In description of mathematical models of the control systems, various mathematical constructions are used, for example, differential and integral equations, linear and nonlinear operators, etc. Note that integral models in some cases have certain advantages over differential ones, since integral models admit as trajectories not only differentiable functions, but also continuous and even integrable functions (see, e.g.  \cite{ban} -- \cite{krasn}). Let us underline that the solutions concepts of the initial and boundary value problems for different type of differential equations, can be expressed via solutions of the appropriate integral equations.

Existence of the optimal trajectories, controllability of the system, necessary and sufficient conditions for optimality of given processes, approximation of the set of trajectories of the control systems described by integral equations are considered in  \cite{ang} -- \cite{hus7} (see also references therein). In  \cite{hus2} -- \cite{hus7} the various properties and approximation of the set of trajectories and integral funnel of the control systems described by Urysohn type integral equations and integral constraints on the control functions are considered. Note that in papers  \cite{hus2} and \cite{hus3}, where approximations of the set of trajectories and of the integral funnel are studied, only the continuous functions are chosen as the trajectories of the system which satisfy the system's equation everywhere.

In this paper, the control system described by the Urysohn type integral equation is investigated. It is assumed that the system is nonlinear with respect to the state vector, and affine with respect to the control vector. The control functions have integral constraint, more precisely, the closed ball of the space $L_2$, centered at the origin with radius $\rho$, is chosen as the set of admissible control functions. The trajectory of the system is defined as multivariable function also from the space $L_2$ satisfying the system's equation almost everywhere. Step by step way, the set of admissible control functions is replaced by a set which consists of a finite number of piecewise-constant control functions and generates a finite number of trajectories. It is proved that in the appropriate settings of discretization parameters, the set consisting of a finite number of trajectories is an internal approximation of the set of trajectories of the control system.

The work is organized as follows. In Section 2,  the basic conditions which satisfies the system's equation are given. In this section, some auxiliary propositions are also presented which are  used in following arguments. In  Section 3, the main result of the paper is formulated and proved, where the set of integrable trajectories of the system is approximated by a set which consists of a finite number of trajectories (Theorem \ref{teo3.1}).

\section{The System's Description}

Consider control system described by the  Ury\-sohn type integral equation
\begin{equation} \label{ue1}
\displaystyle x(\xi)=f\left(\xi,x\left(\xi\right)\right)+\lambda
\int_{E} \left[K_1\left(\xi,s,x\left(s\right)\right)+K_2\left(\xi,s,x\left(s\right)
\right) u\left(s\right)\right] ds
\end{equation}
 where  $x\in \mathbb{R}^n$ is the state vector, $u\in
\mathbb{R}^m$ is the control vector, $\xi \in E,$ $\lambda \in \mathbb{R}^1,$ $E \subset \mathbb{R}^k$ is a compact set. Without loss of generality it will be assumed that $\lambda >0.$

Let $\rho >0$ be a given number,
\begin{equation*}
V_{\rho}=\left\{u(\cdot) \in L_2(E;\mathbb{R}^m):
\left\| u(\cdot) \right\|_2 \leq \rho \right\}
\end{equation*}
where $L_2\left(E;\mathbb{R}^m\right)$ is the space of Lebesgue measurable functions $u(\cdot):E\rightarrow \mathbb{R}^m$ such that
$\left\|u(\cdot)\right\|_2 <+\infty,$ $\displaystyle \left\|u(\cdot)\right\|_2 =\left(\int_{E} \left\| u(s)\right\|^2 ds\right)^{\frac{1}{2}},$
$\left\| \cdot \right\|$ denotes the Euclidean norm.

$V_{\rho}$ is called the set of admissible control functions and every  $u(\cdot) \in V_{\rho}$ is said to be an admissible control function.

It is assumed that the functions and a number $\lambda >0$ given in system (\ref{ue1}) satisfy the following
conditions:

\vspace{1mm}

\textbf{2.A.} The function  $f(\cdot,x ):E\rightarrow
\mathbb{R}^{n}$ is Lebesgue measurable for every fixed $x\in \mathbb{R}^n$, $f(\cdot,0) \in L_2\left(E;\mathbb{R}^n\right)$ and there exists $\gamma_0(\cdot) \in L_{\infty}\left(E;\mathbb{R}^1 \right)$ such that for almost all (a.a.) $\xi \in E$ the inequality
\begin{equation*}
\left\| f(\xi,x_{1})-f(\xi,x_{2})\right\| \leq \gamma_{0}(\xi) \left\| x_{1}-x_{2}\right\|
\end{equation*} is satisfied for every $x_1\in \mathbb{R}^n$ and $x_2\in \mathbb{R}^n$, where $L_{\infty}\big(E;\mathbb{R}^{n_*}\big)$ is the space of Lebesgue measurable functions $w(\cdot):E\rightarrow \mathbb{R}^{n_*}$ such that
$\left\|w(\cdot)\right\|_{\infty} <+\infty,$ $\displaystyle \left\|w(\cdot)\right\|_{\infty} =\inf \left\{ c>0:  \left\| w(s)\right\| \leq c \ \mbox{for a.a.} \ s \in E\right\};$

\vspace{1mm}

\textbf{2.B.} The function  $K_1(\cdot,\cdot,x ):E \times E \rightarrow
\mathbb{R}^{n}$ is Lebesgue measurable for every fixed $x\in \mathbb{R}^n$, $K_1(\cdot,\cdot,0) \in L_2\left(E\times E;\mathbb{R}^n\right)$ and there exists $\gamma_1(\cdot,\cdot) \in L_{2}\left(E \times E;\mathbb{R}^1 \right)$ such that for a.a. $(\xi,s) \in E \times E$ the inequality
\begin{equation*}
\left\| K_1(\xi,s, x_{1})-K_1(\xi,s,x_{2})\right\| \leq \gamma_{1}(\xi,s) \left\| x_{1}-x_{2}\right\|
\end{equation*} is satisfied for every $x_1\in \mathbb{R}^n$ and $x_2\in \mathbb{R}^n$;

\vspace{1mm}

\textbf{2.C.} The function  $K_2(\cdot,\cdot,x ):E \times E \rightarrow
\mathbb{R}^{n\times m}$ is Lebesgue measurable for every fixed $x\in \mathbb{R}^n$, $K_2(\cdot,\cdot,0) \in L_{2}\left(E\times E;\mathbb{R}^{n\times m}\right)$ and there exists $\gamma_2(\cdot,\cdot) \in L_{\infty} (E \times E;\mathbb{R}^{1})$ such that for a.a. $(\xi,s) \in E \times E$ the inequality
\begin{equation*}
\left\| K_2(\xi,s, x_{1})-K_2(\xi,s,x_{2})\right\| \leq \gamma_{2}(\xi,s) \left\| x_{1}-x_{2}\right\|
\end{equation*} is satisfied for every $x_1\in \mathbb{R}^n$ and $x_2\in \mathbb{R}^n$;

\vspace{1mm}

\textbf{2.D.} The inequality
\begin{equation*}
6\left[ \kappa_0^2 +\lambda^2 \kappa_1^2 +\lambda^2 \rho^2 \kappa_2^2 \mu(E)\right]  < 1
\end{equation*}
is satisfied, where $\kappa_0= \left\| \gamma_0(\cdot) \right\|_{\infty}$, $\displaystyle \kappa_1=\left\| \gamma_1(\cdot, \cdot) \right\|_2=\left( \int_{E}\int_{E}\gamma_1(\xi,s)^2 ds \, d\xi \right)^{\frac{1}{2}},$ $\kappa_2 =\left\|\gamma_2(\cdot,\cdot)\right\|_{\infty},$  $\mu(E)$ denotes the Lebesgue measure of the set $E$.

\vspace{1mm}

Since $K_2(\cdot,\cdot,0) \in L_{2}\left(E\times E;\mathbb{R}^{n\times m}\right),$ then according to \cite{kan}, p.318, for $\varepsilon>0$ there exists a continuous function $g_{\varepsilon}(\cdot,\cdot): E \times E \rightarrow \mathbb{R}^{n\times m}$ such that the inequality
\begin{equation}\label{kacon}
\int_{\Omega} \int_{\Omega}\left\| K_2(\xi,s, 0)-g_{\varepsilon}(\xi,s)\right\|^2 ds \, d\xi \leq \varepsilon^{2}
\end{equation}  is satisfied. Denote
\begin{equation}\label{emep}
M(\varepsilon)=\max\left\{ \left\|g_{\varepsilon}(\xi,s)\right\|: (\xi,s) \in E \times E \right\},
\end{equation}
\begin{equation}\label{om*}
\omega_* = \left(\int_E \int_E \left\|K_2(\xi,s,0)\right\|^2 ds \, d\xi\right)^{\frac{1}{2}}.
\end{equation}

Let us define the trajectory of the system (\ref{ue1}) generated by an admissible control function $u(\cdot)\in
V_{\rho}$. A function $x(\cdot) \in L_2\left(E;\mathbb{R}^n\right)$ satisfying the integral equation (\ref{ue1})
for a.a. $\xi\in E$ is said to be a trajectory of the system (\ref{ue1}) generated by the admissible control function
$u(\cdot)\in V_{\rho} \ .$ The set of trajectories of the system (\ref{ue1}) generated by all admissible control functions $u(\cdot)\in V_{\rho}$ is denoted by $\mathbb{Z}_{\rho}$ and is called briefly the set of trajectories of the system (\ref{ue1}).

The following propositions characterize the set of trajectories and will be used in ensuing arguments.
\begin{proposition} \label{prop2.1} \cite{hus7}  Every admissible control function $u(\cdot)\in
V_{\rho}$ generates uni\-que trajectory of the system (\ref{ue1}).
\end{proposition}

\begin{proposition} \label{prop2.2} \cite{hus7}
The set of trajectories $\mathbb{Z}_{\rho}$ of the system (\ref{ue1}) is a bounded subset of the space $L_2\left(E; \mathbb{R}^n\right)$, i.e. there exists $\beta_*>0$ such that $\left\|x(\cdot)\right\|_2 \leq \beta_*$ for every $x(\cdot)\in \mathbb{Z}_{\rho}.$
\end{proposition}

\begin{theorem} \label{teo2.2} \cite{hus7} The set of trajectories $\mathbb{Z}_{\rho}$ of the system (\ref{ue1}) is a compact subset of the space $L_2\left(E; \mathbb{R}^n\right)$.
\end{theorem}

The validity of the following proposition follows from compactness of the set of trajectories $\mathbb{Z}_{\rho}.$
\begin{proposition} \label{prop2.3}
For every $\varepsilon >0$ there exists $\tau_*(\varepsilon)\in \left(0,\frac{\varepsilon^2}{\left[M(\varepsilon)\right]^2}\right]$ such that for any Lebesgue measurable set $E_*\subset E$, where $\mu(E_*)\leq \tau_*(\varepsilon)$, the inequality
\begin{equation*}
\int_{E_*} \left\|z(s)\right\|^2 ds  \leq \varepsilon^2
\end{equation*} is satisfied for every $z(\cdot)\in \mathbb{Z}_{\rho}$. Here $M(\varepsilon)$ is defined by (\ref{emep}).
\end{proposition}

The next propositions will frequently used in proof of the main result of the paper.
\begin{proposition} \label{prop2.4}
Let $y(\cdot) \in L_2\left(E; [0,+\infty)\right),$ $h(\cdot) \in L_2\left(E; [0,+\infty)\right),$ $\psi(\cdot,\cdot) \in L_2\left(E\times E; [0,+\infty)\right),$ $\displaystyle \left\| \psi (\cdot,\cdot)\right\|_2<\frac{1}{\sqrt{2}}$ and
\begin{equation*}
y(\xi) \leq h(\xi)+\int_{E} \psi (\xi,s) y(s) ds
\end{equation*} for a.a. $\xi \in E.$ Then
\begin{equation*}
\left\|y(\cdot)\right\|_2 \leq \left[ \frac{\displaystyle 2\left\|h(\cdot)\right\|_2^2}{\displaystyle 1-2 \left\|\psi (\cdot,\cdot)\right\|_2^2} \right]^{\frac{1}{2}}.
\end{equation*}
\end{proposition}

\begin{proposition} \label{prop2.41}
Let $u(\cdot) \in V_{\rho}$ and
\begin{equation*}
\psi_*(\xi,s) =\frac{\lambda}{1-\kappa_0}\left[ \gamma_1(\xi,s)+\kappa_2 \left\|u(s)\right\|\right]
\end{equation*} for a.a. $(\xi,s) \in E\times E.$ Then
\begin{equation*}
\displaystyle \left\| \psi_* (\cdot,\cdot)\right\|_2 \leq \frac{\sqrt{2} \lambda}{1-\kappa_0} \left[\kappa_1^2+\rho^2 \kappa_2^2\mu(E)\right]^{\frac{1}{2}} <\frac{1}{\sqrt{2}}.
\end{equation*}
\end{proposition}

\begin{proposition} \label{prop2.42}
Let $\beta_1>0,$ $\beta_2>0$ and
\begin{equation*}
h_*(\xi) =\beta_1+\beta_2 \left( \int_E  \left\|K_2(\xi,s,0)\right\|^2ds \right)^{\frac{1}{2}}
\end{equation*} for a.a. $\xi \in E.$ Then
\begin{equation*}
\displaystyle \left\| h_* (\cdot)\right\|_2 \leq  \left[2\beta_1^2\mu(E)+2\beta_2^2 \omega_*^2\right]^{\frac{1}{2}}
\end{equation*} where $\omega_*$ is defined by (\ref{om*}).
\end{proposition}

Now let us give definition of finite $\Delta$-partition of the set $Q\subset \mathbb{R}^{k}.$
\begin{definition} \label{def2.1} Let $Q\subset \mathbb{R}^{k}$ be a given set. A finite system of sets
$\Gamma = \{Q_1, Q_2, \ldots, Q_g\}$ is said to be a finite $\Delta$-partition of given $Q$ if

\vspace{2mm}

$\mathbf{2.d_1.}$ $Q_i \subset Q$ and $Q_i$ is Lebesgue measurable for every $i=1,2,\ldots ,g$;

$\mathbf{2.d_2.}$ $Q_i\bigcap Q_j =\emptyset$ for every $i\neq j$, where $i=1,2,\ldots, g$ and $j=1,2,\ldots, g$;

$\mathbf{2.d_3.}$ $Q =\bigcup_{i=1}^{g} Q_i$;

$\mathbf{2.d_4.}$ $diam \, (Q_i) \leq \Delta$ for every $i=1,2,\ldots, g,$ where $diam \, (Q_i)=\sup\big\{\left\|x-y\right\|:$ $x\in Q_i, \ y\in Q_i \big\}$.
\end{definition}
\begin{proposition}\label{prop2.5}
Let $Q\subset \mathbb{R}^{k}$ be a compact set. Then for every $\Delta>0$ it has a finite $\Delta$-partition.
\end{proposition}

\section{Approximation of the Set of Trajectories}

Let $\alpha>0,$ $\Delta >0$ and $\sigma>0$ be given numbers, $\Gamma=\left\{E_1,E_2,\ldots, E_N\right\}$ be a finite $\Delta$-partition of the compact set $E\subset \mathbb{R}^k,$ $\Lambda=\left\{0=r_0,r_1,\ldots, r_q=\alpha \right\}$ be a uniform partition of the closed interval $[0,\alpha]$ with diameter $\delta=r_{j+1}-r_j,$ $j=0,1,\ldots, q-1,$  $S=\left\{x\in \mathbb{R}^m: \left\|x\right\|=1\right\},$ $S_{\sigma}=\left\{b_1,b_2, \ldots, b_c\right\}$ be a finite $\sigma$-net on the compact $S.$ Denote

\begin{align}\label{fincon}
V_{\rho}^{\alpha, \Gamma, \Lambda,\sigma} &= \big\{ u(\cdot)\in L_2(E; \mathbb{R}^m): u(s)=r_{j_i}b_{l_i} \ \mbox{for every} \ s \in E_i, \nonumber \\ & \qquad r_{j_i} \in \Lambda, \ b_{l_i}\in S_{\sigma},  \ i=1,2,\ldots, N, \ \sum_{i=1}^{N} \mu(E_i)r_{j_i}^2 \leq \rho^2 \big\}.
\end{align}

By symbol $\mathbb{Z}_{\rho}^{\alpha,\Gamma,\Lambda,\sigma}$ we denote the set of trajectories of the system (\ref{ue1}) generated by all control functions $u(\cdot)\in V_{\rho}^{\alpha,\Gamma,\Lambda,\sigma}.$ It is obvious that the set $V_{\rho}^{\alpha, \Gamma, \Lambda,\sigma}$ consists of a finite number of piecewise-constant functions, the set $\mathbb{Z}_{\rho}^{\alpha, \Gamma, \Lambda,\sigma}$ consists of a finite number of trajectories.

Let
\begin{equation}\label{ce*}
c_*=\left[ \frac{\displaystyle 16 \lambda^2 \rho^2 \left[ \left(\kappa_2+1\right)^2 \mu(E)+1\right]}{(1-\kappa_0)^2- 4\lambda^2 \left[\kappa_1^2+\rho^2 \kappa_2^2 \mu(E)\right] } \right]^{\frac{1}{2}}.
\end{equation}

\begin{theorem}\label{teo3.1}
For every $\varepsilon >0$ there exist $\alpha_*(\varepsilon)>0$, $\Delta_*(\varepsilon)>0$, $\delta_*(\varepsilon)>0$, $\sigma_*(\varepsilon,\alpha_*(\varepsilon))>0$ such that for any $\Delta \in \left(0,\Delta_*(\varepsilon)\right],$ $\delta \in \left(0,\delta_*(\varepsilon)\right]$ and $\sigma \in (0,\sigma_*(\varepsilon,\alpha_*(\varepsilon))]$ the inequality
\begin{equation}\label{os1}
h_{2}\left(\mathbb{Z}_{p,r}, \mathbb{Z}_{p,r}^{\alpha_*(\varepsilon),\Gamma, \Lambda,\sigma}\right) \leq (c_*+1) \varepsilon
\end{equation}
is satisfied where  $\Gamma=\left\{E_1,E_2,\ldots, E_N\right\}$ is a finite $\Delta$-partition of the compact set  $E,$ \linebreak $\Lambda=\left\{0=r_0,r_1,\ldots, r_q=\alpha_*(\varepsilon) \right\}$ is a uniform partition of the closed interval  $[0,\alpha_*(\varepsilon)]$, $\delta=r_{j+1}-r_j$, $j=0,1,\ldots,q-1$, is the diameter of partition $\Lambda$, $h_2(\cdot,\cdot)$ stands for the Hausdorff distance between the subsets of the space $L_2\left(E; \mathbb{R}^n\right)$, $c_*$ is defined by \eqref{ce*}.
\end{theorem}
\begin{proof}  For given $\alpha>0$ we set
\begin{equation*}
V_{\rho}^{\alpha}=\left\{u(\cdot) \in V_{\rho}: \left\| u(s) \right\| \leq \alpha \ \mbox{for a.a.} \ s\in E \right\} \, ,
\end{equation*} and let $\mathbb{Z}_{\rho}^{\alpha}$ be the set of trajectories of the system (\ref{ue1}) generated by all control functions $u(\cdot)\in V_{\rho}^{\alpha}.$

Denote
\begin{equation}\label{alfa*}
\alpha_*(\varepsilon)=\frac{\rho}{\tau_*(\varepsilon)^{\frac{1}{2}}}
\end{equation} where $\tau_*(\varepsilon)$ is defined in Proposition \ref{prop2.3}.

Let $\alpha\geq \alpha_*(\varepsilon)$ be fixed. Choose an arbitrary $z(\cdot) \in \mathbb{Z}_{\rho}$ generated by the control function $v(\cdot) \in V_{\rho}.$ Define new control function $v_*(\cdot):E\rightarrow \mathbb{R}^m$, setting
\begin{eqnarray}\label{eq2}
v_*(s)=\left\{
\begin{array}{llll}
v(s)  & \mbox{if} & \left\|v(s)\right\| \leq \alpha, \\
\displaystyle \frac{v(s)}{\left\|v(s)\right\|} \alpha  & \mbox{if} & \left\|v(s)\right\| > \alpha
\end{array}
\right.
\end{eqnarray}
It is obvious that $v_*(\cdot) \in V_{\rho}^{\alpha}.$ Let $z_*(\cdot):E\rightarrow \mathbb{R}^n$ be the trajectory of the system (\ref{ue1}) generated by the control function $v_*(\cdot)$ and $E_*=\left\{s\in E: \left\|v(s)\right\|>\alpha\right\}.$ Then $z_*(\cdot) \in \mathbb{Z}_{\rho}^{\alpha}.$ From Tchebyshev's inequality (see, \cite{whe}, p.82) it follows that
\begin{equation} \label{eq3}
\mu(E_*) \leq \frac{\rho^2}{\alpha^2}.
\end{equation}

Since $\alpha \geq \alpha_*(\varepsilon),$ then from (\ref{alfa*}) and (\ref{eq3}) it follows that
\begin{equation} \label{mue*}
\mu(E_*) \leq \frac{\rho^2}{\alpha^2} \leq \tau_*(\varepsilon) .
\end{equation}

On behalf of conditions 2.A, 2.B, 2.C and (\ref{eq2}) we have
\begin{align}\label{eq4}
\left\| z(\xi)-z_*(\xi) \right\| & \leq  \kappa_0 \cdot \left\|z(\xi)-z_*(\xi)\right\|   +\lambda \int_{E} \big[\gamma_1(\xi,s)+ \kappa_2\left\|v(s)\right\|\big]\cdot \left\| z(s)-z_*(s) \right\| ds \nonumber \\ & \quad +\lambda \int_{E_*} \left\|K_2(\xi,s,z_*(s))\right\| \cdot \left\| v(s)-v_*(s)\right\| ds
\end{align} for a.a. $\xi \in E.$

Condition 2.C, (\ref{emep}), (\ref{mue*}),  inclusions $v(\cdot) \in V_{\rho},$ $v_*(\cdot) \in V_{\rho},$ $z_*(\cdot) \in \mathbb{Z}_{\rho},$ Proposition \ref{prop2.3} and Cauchy-Schwarz inequality  imply that
\begin{align}\label{eq6}
& \int_{E_*} \left\|K_2(\xi,s,z_*(s))\right\| \cdot \left\| v(s)-v_*(s)\right\| ds  \leq  \kappa_2 \int_{E_*} \left\|z_*(s)\right\| \cdot \left\| v(s)-v_*(s)\right\| ds \nonumber \\ & \qquad + \int_{E_*} \left\|K_2(\xi,s,0)-g_{\varepsilon}(\xi,s)\right\| \cdot \left\| v(s)-v_*(s)\right\| ds + \int_{E_*} \left\|g_{\varepsilon}(\xi,s)\right\| \cdot \left\| v(s)-v_*(s)\right\| ds \nonumber \\
& \quad \leq  2\rho \kappa_2\varepsilon  +2\rho M(\varepsilon) \left[\mu(E_*)\right]^{\frac{1}{2}} + 2\rho \left( \int_{E} \left\|K_2(\xi,s,0)-g_{\varepsilon}(\xi,s)\right\|^2 ds \right)^{\frac{1}{2}} \nonumber \\
& \quad \leq 2\rho \left( \kappa_2+1\right)\varepsilon  + 2\rho \left( \int_{E} \left\|K_2(\xi,s,0)-g_{\varepsilon}(\xi,s)\right\|^2 ds \right)^{\frac{1}{2}}
\end{align} for a.a. $\xi \in E.$ On behalf of the Condition 2.D we have that $\kappa_0<1.$ Then, from (\ref{eq4}) and (\ref{eq6}) it follows that
\begin{align}\label{eq7}
& \left\| z(\xi)-z_*(\xi) \right\|  \leq   \frac{\lambda}{1-\kappa_0} \int_{E} \big[\gamma_1(\xi,s)+ \kappa_2\left\|v(s)\right\|\big]\cdot \left\| z(s)-z_*(s) \right\| ds \nonumber \\ & \qquad  + \frac{2\lambda \rho \left(\kappa_2+1\right)}{1-\kappa_0}\varepsilon +\frac{2\lambda \rho}{1-\kappa_0} \left( \int_{E} \left\|K_2(\xi,s,0)-g_{\varepsilon}(\xi,s)\right\|^2 ds \right)^{\frac{1}{2}}
\end{align}
for a.a. $\xi \in E.$ Denote
\begin{equation}\label{eq8}
\psi_1(\xi,s)= \frac{\lambda}{1-\kappa_0}  \big[\gamma_1(\xi,s)+ \kappa_2\left\|v(s)\right\| \big],
\end{equation}
for a.a. $(\xi,s)\in E\times E$ and
\begin{equation}\label{eq9}
h_1(\xi)= \frac{2\lambda \rho \left(\kappa_2+1\right)}{1-\kappa_0}\varepsilon +\frac{2\lambda \rho}{1-\kappa_0} \left( \int_{E} \left\|K_2(\xi,s,0)-g_{\varepsilon}(\xi,s)\right\|^2 ds \right)^{\frac{1}{2}}
\end{equation} for a.a. $\xi \in E.$ Since $v(\cdot) \in V_{\rho}^{\alpha}\subset V_{\rho}$, then (\ref{eq8}) and Proposition \ref{prop2.41} imply that
\begin{equation}\label{eq11}
\displaystyle \left\|\psi_1(\cdot,\cdot)\right\|_2  \leq
\frac{\sqrt{2}\lambda}{1-\kappa_0}  \left[\kappa_1^2+\rho^2 \kappa_2^2 \mu(E)\right]^{\frac{1}{2}}< \frac{1}{\sqrt{2}}.
\end{equation}

By virtue of (\ref{eq9}) have that
\begin{equation*}\label{ax1}
h_1(\xi)^2 \leq \frac{8\lambda^2 \rho^2 \left(\kappa_2+1\right)^2}{(1-\kappa_0)^2}\varepsilon^2 +\frac{8 \lambda^2 \rho^2}{(1-\kappa_0)^2} \int_{E} \left\|K_2(\xi,s,0)-g_{\varepsilon}(\xi,s)\right\|^2 ds
\end{equation*} for a.a. $\xi \in E$. Integrating the last inequality on the set $E$ and taking into consideration the inequality  (\ref{kacon}) we obtain
\begin{equation}\label{eq11*}
\left\|h_1(\cdot)\right\|_2^2 \leq  \frac{8\lambda^2 \rho^2 \left[ \left(\kappa_2+1\right)^2 \mu(E)+1\right]}{(1-\kappa_0)^2} \cdot \varepsilon^2 .
\end{equation}

Now, from   (\ref{eq7}), (\ref{eq8}), (\ref{eq9}), (\ref{eq11}), (\ref{eq11*}) and Proposition \ref{prop2.4} we have
\begin{equation}\label{eq12}
\left\|z(\cdot)-z_*(\cdot)\right\|_2  \leq  \left[ \frac{\displaystyle 16 \lambda^2 \rho^2 \left[ \left(\kappa_2+1\right)^2 \mu(E)+1\right]}{(1-\kappa_0)^2- 4\lambda^2 \left[\kappa_1^2+\rho^2 \kappa_2^2 \mu(E)\right] } \right]^{\frac{1}{2}} \cdot \varepsilon =c_*\varepsilon
\end{equation} where $c_*$ is defined by (\ref{ce*}). Since $z(\cdot) \in \mathbb{Z}_{\rho}$ is an arbitrarily chosen trajectory, $z_*(\cdot) \in \mathbb{Z}_{\rho}^{\alpha},$ then (\ref{eq12}) implies that
\begin{equation}\label{eq13}
\mathbb{Z}_{\rho}  \subset \mathbb{Z}_{\rho}^{\alpha} +c_*\varepsilon \cdot \mathcal{B}_2(1)
\end{equation} where
\begin{equation}\label{be2}
\mathcal{B}_2(1)=\left\{z(\cdot)\in L_2(E;\mathbb{R}^n): \left\|z(\cdot)\right\|_2 \leq 1 \right\}.
\end{equation}

(\ref{eq13}) and inclusion $\mathbb{Z}_{\rho}^{\alpha}  \subset \mathbb{Z}_{\rho}$ imply that for every $\alpha \geq \alpha_*(\varepsilon)$ the inequality
\begin{equation*}
h_{2}\left(\mathbb{Z}_{p,r}, \mathbb{Z}_{p,r}^{\alpha}\right) \leq  c_*\varepsilon
\end{equation*} is satisfied. In particular, we have that the inequality
\begin{equation}\label{eq14}
h_{2}\left(\mathbb{Z}_{p,r}, \mathbb{Z}_{p,r}^{\alpha_*(\varepsilon)}\right) \leq  c_*\varepsilon
\end{equation} holds where $\alpha_*(\varepsilon)$ is defined by (\ref{alfa*}).

Denote
\begin{equation}\label{g1}
g_1 = \left[ \frac{\displaystyle 4\lambda^2 \left[ \kappa_2^2 \beta_*^2 \mu(E)+\omega_*^2\right]}{(1-\kappa_0)^2- 4\lambda^2 \left[\kappa_1^2+\rho^2 \kappa_2^2 \mu(E)\right]} \right]^{\frac{1}{2}}.
\end{equation}

Now we narrow down the set of control functions $V_{\rho}^{\alpha_*(\varepsilon)}$ and define new set of control functions which are Lipschitz continuous and satisfy integral and geometric constraints. We set
\begin{equation}\label{eq15}
V_{\rho}^{\alpha_*(\varepsilon),Lip}=\left\{u(\cdot) \in V_{\rho}^{\alpha_*(\varepsilon)}:  u(\cdot): E \rightarrow \mathbb{R}^m \  \mbox{is Lipschitz continuous}  \right\} \, ,
\end{equation} and let $\mathbb{Z}_{\rho}^{\alpha_*(\varepsilon),Lip}$ be the set of trajectories of the system (\ref{ue1}) generated by all control functions $u(\cdot)\in V_{\rho}^{\alpha_*(\varepsilon),Lip}.$ It will be proved that
\begin{equation}\label{eq16}
h_2\left( \mathbb{Z}_{\rho}^{\alpha_*(\varepsilon)}, \mathbb{Z}_{\rho}^{\alpha_*(\varepsilon),Lip} \right)=0 \, .
\end{equation}

Let us choose an arbitrary $\nu>0$ and $y(\cdot) \in \mathbb{Z}_{\rho}^{\alpha_*(\varepsilon)}$ which is generated by the control function $w(\cdot) \in V_{\rho}^{\alpha_*(\varepsilon)}.$ Now let $\eta_j\in (0,1)$ for every $j=1,2,\ldots$ and $\eta_j \rightarrow 0$ as $j\rightarrow +\infty.$ For each $j=1,2,\ldots$ define functions $w_j(\cdot):E \rightarrow \mathbb{R}^m$, setting
\begin{equation}\label{eq17}
w_{j}(s)= \frac{1}{v_{j}} \int_{B_k(s,\eta_j)}w(\tau)d\tau \, , \ s\in E
\end{equation} where $B_k(s,\eta_j)=\left\{\zeta \in \mathbb{R}^k:\left\|\zeta-s\right\|<\eta_j\right\}$, $v_j$ is the Lebesgue measure of the open ball centered at the origin with radius $\eta_j$ in the space $\mathbb{R}^k$, i.e. $v_j =\mu (B_k(0,\eta_j)).$  If $\tau \not \in E$, then in the equality (\ref{eq17}) it is assumed that $w(\tau)=0.$ The function $w_j(\cdot)$ is called Steklov average of the function $w(\cdot)$ for $\eta_j$. According to \cite{kan} (Lemma 1, p.317) we have that $w_j(\cdot) \in V_{\rho}^{\alpha_*(\varepsilon),Lip}$ for every $j=1,2,\ldots$ and $\left\|w_j(\cdot)-w(\cdot)\right\|_2 \rightarrow 0$ as $j\rightarrow +\infty.$ Thus, we obtain that for given $\nu>0$ there exists $j_*$ such that
\begin{equation}\label{eq18}
\left\| w(\cdot)-w_{j_*}(\cdot)\right\|_2 \leq \nu .
\end{equation}

Let $y_*(\cdot)$ be the trajectory of the system (\ref{ue1}) generated by the control function $w_{j_*}(\cdot) \in V_{\rho}^{\alpha_*(\varepsilon),Lip}.$ Then $y_{*}(\cdot) \in \mathbb{Z}_{\rho}^{\alpha_*(\varepsilon),Lip}.$ From Conditions 2.A, 2.B, 2.C, 2.D, (\ref{eq18}), Proposition \ref{prop2.2} and Caushy-Schwarz  inequality it follows that
\begin{align*}
& \left\| y(\xi)-y_*(\xi) \right\|  \leq \kappa_0 \cdot \left\|y(\xi)-y_*(\xi)\right\| +\lambda \int_{E} \big[\gamma_1(\xi,s)+ \kappa_2\left\|w(s)\right\|\big]\cdot \left\| y(s)-y_*(s) \right\| ds \nonumber \\ & \qquad +\lambda \int_{E} \left[\kappa_2 \left\|y_*(s)\right\| +\left\|K_2(\xi,s,0)\right\| \right] \left\| w(s)-w_{j_*}(s)\right\| ds \nonumber \\
& \quad \leq \kappa_0 \cdot \left\|y(\xi)-y_*(\xi)\right\| +\lambda \int_{E} \big[\gamma_1(\xi,s)+ \kappa_2\left\|w(s)\right\|\big]\cdot \left\| y(s)-y_*(s) \right\| ds \nonumber \\ & \qquad +\lambda \kappa_2 \beta_* \nu  +\lambda \nu \left( \int_E \left\|K_2(\xi,s,0)\right\|^2 ds  \right)^{\frac{1}{2}} \, ,
\end{align*} and hence
\begin{align}\label{eq20}
\| y(\xi)-y_*(\xi) \| & \leq  \frac{\lambda}{1-\kappa_0} \int_{E} \big[\gamma_1(\xi,s)+ \kappa_2 \|w(s)\|\big]\cdot \| y(s)-y_*(s) \| ds \nonumber \\ & \qquad +  \frac{\lambda \kappa_2 \beta_*\nu}{1-\kappa_0} +\frac{\lambda \nu}{1-\kappa_0} \left( \int_E \|K_2(\xi,s,0) \|^2 ds \right)^{\frac{1}{2}}
\end{align} for a.a. $\xi \in E$.

Denote
\begin{equation}\label{eq21}
\psi_2(\xi,s)= \frac{\lambda}{1-\kappa_0}  \big[\gamma_1(\xi,s)+ \kappa_2\left\|w(s)\right\| \big], \ \  \mbox{a.a.} \ (\xi,s)\in E\times E,
\end{equation}
\begin{equation}\label{eq22}
h_2(\xi)= \frac{\lambda \kappa_2 \beta_*\nu}{1-\kappa_0} +\frac{\lambda \nu}{1-\kappa_0} \left( \int_E \left\|K_2(\xi,s,0)\right\|^2 ds \right)^{\frac{1}{2}}, \ \ \mbox{a.a.} \ \xi \in E.
\end{equation}

Since $w(\cdot)\in V_{\rho}^{\alpha_*(\varepsilon)} \subset V_{\rho}$, then (\ref{eq21}) and Proposition \ref{prop2.41} imply
\begin{equation}\label{eq23}
\displaystyle \left\|\psi_2(\cdot,\cdot)\right\|_2  \leq
\frac{\sqrt{2}\lambda}{1-\kappa_0}  \left[\kappa_1^2+\rho^2 \kappa_2^2 \mu(E)\right]^{\frac{1}{2}} <\frac{1}{\sqrt{2}}.
\end{equation}

Now, from (\ref{om*}), (\ref{eq22}) and  Proposition \ref{prop2.42} it follows that
\begin{equation}\label{eq24}
\left\|h_2(\cdot)\right\|_2^2  \leq  2 \left( \frac{\lambda \kappa_2 \beta_*\left[\mu(E)\right]^{\frac{1}{2}}}{1-\kappa_0}\right)^2 \nu^2 +2\left(\frac{\lambda \omega_*}{1-\kappa_0}\right)^2 \nu^2 \, .
\end{equation}

From  (\ref{g1}), (\ref{eq20}), (\ref{eq21}), (\ref{eq22}), (\ref{eq23}), (\ref{eq24}) and Proposition \ref{prop2.4} we have
\begin{equation}\label{eq26}
\left\|y(\cdot)-y_*(\cdot)\right\|_2  \leq  \left[ \frac{\displaystyle 4 \lambda^2 \kappa_2^2 \beta_*^2 \mu(E)+4\lambda^2\omega_*^2}{(1-\kappa_0)^2- 4\lambda^2 \left[\kappa_1^2+\rho^2 \kappa_2^2 \mu(E)\right] } \right]^{\frac{1}{2}} \nu =g_1\nu \, .
\end{equation}

Since $y(\cdot) \in \mathbb{Z}_{\rho}^{\alpha_*(\varepsilon)}$ is an arbitrarily chosen trajectory and $y_*(\cdot) \in \mathbb{Z}_{\rho}^{\alpha_*(\varepsilon),Lip},$ then  (\ref{eq26}) yields that
\begin{equation}\label{eq27}
\mathbb{Z}_{\rho}^{\alpha_*(\varepsilon)} \subset \mathbb{Z}_{\rho}^{\alpha_*(\varepsilon),Lip} +g_1 \nu \cdot \mathcal{B}_2(1)
\end{equation} where $\mathcal{B}_2(1)$ is defined by (\ref{be2}), $g_1$ is defined by (\ref{g1}).

From inclusions $\mathbb{Z}_{\rho}^{\alpha_*(\varepsilon),Lip} \subset \mathbb{Z}_{\rho}^{\alpha_*(\varepsilon)}$ and (\ref{eq27}) it follows that
\begin{equation*}
h_2\left( \mathbb{Z}_{\rho}^{\alpha_*(\varepsilon)}, \mathbb{Z}_{\rho}^{\alpha_*(\varepsilon),Lip} \right) \leq g_1\nu \, .
\end{equation*} Since $\nu >0$ is an arbitrarily chosen number, the last inequality implies the validity of the equality (\ref{eq16}).

Now, from (\ref{eq14}) and (\ref{eq16}) we conclude that the inequality
\begin{equation}\label{eq27a}
h_2\left( \mathbb{Z}_{\rho}, \mathbb{Z}_{\rho}^{\alpha_*(\varepsilon),Lip} \right) \leq c_*\varepsilon
\end{equation} holds, where $c_*$ and $\alpha_*(\varepsilon)$ are defined by (\ref{ce*}) and (\ref{alfa*}) respectively.

Define new set of control functions. For given $\alpha_*(\varepsilon)>0$ and integer $R>0$ let us denote
\begin{equation}\label{eq28}
V_{\rho}^{\alpha_*(\varepsilon), Lip, R} = \big\{ u(\cdot)\in V_{\rho}^{\alpha_*(\varepsilon), Lip} : \ \mbox{the Lipschitz constant of} \  u(\cdot) \ \mbox{is not greater than} \ R \big\} \, ,
\end{equation} and let $\mathbb{Z}_{\rho}^{\alpha_*(\varepsilon),Lip,R}$ be the set of trajectories of the system (\ref{ue1}) generated by all control functions $u(\cdot)\in V_{\rho}^{\alpha_*(\varepsilon),Lip,R}$ where the set $V_{\rho}^{\alpha_*(\varepsilon),Lip}$ is defined by (\ref{eq15}).

One can show that $V_{\rho}^{\alpha_*(\varepsilon), Lip, R}\subset C(E;\mathbb{R}^m)$ and $\mathbb{Z}_{\rho}^{\alpha_*(\varepsilon), Lip, R} \subset L_2(E;\mathbb{R}^n)$
are compact sets, where $C(E;\mathbb{R}^m)$ is the space of continuous functions $u(\cdot):E\rightarrow \mathbb{R}^m$ with norm $\left\|u(\cdot)\right\|_C= \max \left\{ \left\|u(s)\right\|:s\in E \right\}$.

It is not difficult to prove that  $V_{\rho}^{\alpha_*(\varepsilon), Lip} = \displaystyle \bigcup_{R=1}^{+\infty}V_{\rho}^{\alpha_*(\varepsilon), Lip, R}$, and hence
\begin{equation}\label{eq29}
\mathbb{Z}_{\rho}^{\alpha_*(\varepsilon),Lip} = \bigcup_{R=1}^{+\infty} \mathbb{Z}_{\rho}^{\alpha_*(\varepsilon), Lip,R}
\end{equation} where $V_{\rho}^{\alpha_*(\varepsilon), Lip, R}$ is defined by (\ref{eq28}). (\ref{eq27a}) and (\ref{eq29}) imply that the inequality
\begin{equation*}
h_2\left( \mathbb{Z}_{\rho}, \bigcup_{R=1}^{+\infty} \mathbb{Z}_{\rho}^{\alpha_*(\varepsilon),Lip,R} \right) \leq c_*\varepsilon
\end{equation*} holds, where $c_*$ and $\alpha_*(\varepsilon)$ are defined by (\ref{ce*}) and (\ref{alfa*}) respectively. According to the Theorem \ref{teo2.2}, the set of trajectories $\mathbb{Z}_{\rho}$ of the system (\ref{ue1}) is a compact subset of the space $L_2(E;\mathbb{R}^n)$. Since
\begin{equation*}
\mathbb{Z}_{\rho}^{\alpha_*(\varepsilon),Lip,R} \subset \mathbb{Z}_{\rho}^{\alpha_*(\varepsilon),Lip,R+1}\subset \mathbb{Z}_{\rho}
\end{equation*}
for every $R=1,2,\ldots$, then we have that there exists $R_*(\varepsilon)>0$ such that for every $R\geq R_*(\varepsilon)$ the inequality
\begin{equation*}
h_2\left( \mathbb{Z}_{\rho}, \mathbb{Z}_{\rho}^{\alpha_*(\varepsilon),Lip,R} \right) \leq \left(c_*+\frac{1}{4}\right)\varepsilon
\end{equation*} is satisfied. In particular, the last inequality yields that the inequality
\begin{equation}\label{eq30}
h_2\left( \mathbb{Z}_{\rho}, \mathbb{Z}_{\rho}^{\alpha_*(\varepsilon),Lip,R_*(\varepsilon)} \right) \leq \left(c_*+\frac{1}{4}\right)\varepsilon
\end{equation} is held.

We introduce new set of control functions which consists of piecewise-constant functions satisfying the integral and geometric constraints. For given $\Delta>0$ and a finite $\Delta$-partition $\Gamma=\left\{E_1,E_2,\ldots, E_N\right\}$ of the compact set $E$ we denote
\begin{equation}\label{eq31}
V_{\rho}^{\alpha_*(\varepsilon), \Gamma}= \big\{ u(\cdot)\in V_{\rho}^{\alpha_*(\varepsilon)} : u(s)=u_i \ \mbox{for every} \ s \in E_i, \ i=1,2,\ldots, N \big\}.
\end{equation} By symbol $\mathbb{Z}_{\rho}^{\alpha_*(\varepsilon),\Gamma}$ we denote the set of trajectories of the system (\ref{ue1}) generated by all control functions $u(\cdot)\in V_{\rho}^{\alpha_*(\varepsilon),\Gamma}$, and let
\begin{equation}\label{g2}
g_2=\left[\displaystyle \frac{4\lambda^2 \left[\kappa_2^2 \beta_*^2 + \omega_*^2\right]\mu(E)}{\displaystyle (1-\kappa_0)^2-4\lambda^2 \left[\kappa_1^2+\rho^2 \kappa_2^2 \mu(E)\right]^2}\right]^{\frac{1}{2}}.
\end{equation}

It will be proved that the inclusion
\begin{equation}\label{eq32}
\mathbb{Z}_{\rho}^{\alpha_*(\varepsilon),Lip, R_*(\varepsilon)} \subset \mathbb{Z}_{\rho}^{\alpha_*(\varepsilon),\Gamma} +g_2 R_*(\varepsilon) \Delta \cdot \mathcal{B}_2(1)
\end{equation} is held where $\mathcal{B}_2(1)$ is defined by (\ref{be2}), $g_2$ is defined by (\ref{g2}), $ R_*(\varepsilon)$ is defined in (\ref{eq30}).

Choose an arbitrary $\tilde{x}(\cdot)\in \mathbb{Z}_{\rho}^{\alpha_*(\varepsilon),Lip, R_*(\varepsilon)}$ generated by the control function $\tilde{u}(\cdot)\in V_{\rho}^{\alpha_*(\varepsilon),Lip, R_*(\varepsilon)}.$ According to (\ref{eq28}) we have
\begin{eqnarray}\label{eq33}
\left\{
\begin{array}{llll}
\left\| \tilde{u}(\cdot)\right\|_2 \leq \rho, \ \ \left\| \tilde{u}(s)\right\| \leq \alpha_*(\varepsilon) \ \mbox{for every} \  s\in E, \\
\left\| \tilde{u}(s_1)-\tilde{u}(s_2) \right\| \leq R_*(\varepsilon)\left\| s_1-s_2 \right\| \ \mbox{for every} \ s_1 \in E, \ s_2 \in E.
\end{array}
\right.
\end{eqnarray}

Define new control function $\tilde{u}_*(\cdot):E\rightarrow \mathbb{R}^m$, setting
\begin{equation}\label{eq34}
\tilde{u}_*(s)=\frac{1}{\mu(E_i)} \int_{E_i} \tilde{u}(\tau) d\tau, \ \ s\in E_i, \ i=1,2,\ldots, N.
\end{equation}
It is obvious that the function $\tilde{u}_*(\cdot):E\rightarrow \mathbb{R}^m$ is constant on the set $E_i,$ $i=1,2,\ldots, N,$ and $\left\| \tilde{u}_*(s)\right\| \leq \alpha_*(\varepsilon)$ for every $s\in E.$ From (\ref{eq34}) and Chauchy-Schwarz inequality it follows that
\begin{equation*}\label{eq35}
\int_{E_i}\left\|\tilde{u}_*(s)\right\|^2 ds \leq \int_{E_i} \left\|\tilde{u}(\tau)\right\|^2 d\tau
\end{equation*} which implies that $\left\|\tilde{u}_*(\cdot)\right\|_2 \leq \left\|\tilde{u}(\cdot)\right\|_2  \leq \rho.$
So, we have that $\tilde{u}_*(\cdot)\in V_{\rho}^{\alpha_*(\varepsilon),\Gamma}.$ Let $\tilde{x}_*(\cdot)$ be the trajectory of the system (\ref{ue1}) generated by the control function $\tilde{u}_*(\cdot).$ Then $\tilde{x}_*(\cdot)\in \mathbb{Z}_{\rho}^{\alpha_* (\varepsilon), \Gamma}.$

Let us choose an arbitrary $s\in E$ and fix it. By virtue of Definition \ref{def2.1} there exists $i_*=1,2,\ldots, N$ such that $s\in E_{i_*}.$ From (\ref{eq33}), (\ref{eq34}) and inequality $diam \, (E_{i_*}) \leq \Delta$ it follows that
\begin{equation*}
\left\| \tilde{u}(s)-\tilde{u}_*(s)\right\| \leq \frac{1}{\mu(E_{i_*})} \int_{E_{i_*}}\left\| \tilde{u}(s)- \tilde{u}(\tau)\right\| d\tau  \leq R_*(\varepsilon)\Delta.
\end{equation*}

Since $s\in E$ is arbitrarily fixed, then we have that
\begin{equation}\label{eqlip}
\left\| \tilde{u}(s)-\tilde{u}_*(s)\right\| \leq  R_*(\varepsilon)\Delta
\end{equation} for every $s\in E.$ Now, from Conditions  2.A, 2.B, 2.C, 2.D, Cauchy-Schwarz inequality, Proposition \ref{prop2.2} and (\ref{eqlip})  we obtain
\begin{align}\label{eq36}
& \left\| \tilde{x}(\xi)-\tilde{x}_*(\xi) \right\| \leq  \frac{\lambda}{1-\kappa_0} \int_{E} \big[\gamma_1(\xi,s)+ \kappa_2\left\|\tilde{u}(s)\right\|\big]\cdot \left\| \tilde{x}(s)-\tilde{x}_*(s) \right\| ds \nonumber \\ & \qquad + \frac{\lambda \kappa_2 \beta_* \left[\mu(E)\right]^{\frac{1}{2}} R_*(\varepsilon)\Delta}{1-\kappa_0} +\frac{\lambda \left[\mu(E)\right]^{\frac{1}{2}} R_*(\varepsilon)\Delta}{1-\kappa_0} \left( \int_E \left\|K_2(\xi,s,0)\right\|^2 ds \right)^{\frac{1}{2}}
\end{align} for a.a. $\xi \in E.$ Since $\tilde{u}(\cdot)\in V_{\rho}^{\alpha_*(\varepsilon),Lip, R_*(\varepsilon)} \subset V_{\rho},$ then (\ref{g2}), (\ref{eq36}), Propositions \ref{prop2.4}, \ref{prop2.41} and \ref{prop2.42} imply that
\begin{equation}\label{eq41}
\left\| \tilde{x}(\cdot)-\tilde{x}_*(\cdot) \right\|_2  \leq \left[\displaystyle \frac{4\lambda^2 \left[\kappa_2^2 \beta_*^2 + \omega_*^2\right]\mu(E)}{\displaystyle (1-\kappa_0)^2-4\lambda^2 \left[\kappa_1^2+\rho^2 \kappa_2^2 \mu(E)\right]}\right]^{\frac{1}{2}} \cdot R_*(\varepsilon)\Delta =g_2 R_*(\varepsilon)\Delta.
\end{equation}

Since $\tilde{x}(\cdot)\in \mathbb{Z}_{\rho}^{\alpha_*(\varepsilon),Lip, R_*(\varepsilon)}$ is arbitrarily chosen trajectory, $\tilde{x}_*(\cdot)\in \mathbb{Z}_{\rho}^{\alpha_*(\varepsilon),\Gamma},$ then (\ref{eq41}) yields the proof of the inclusion (\ref{eq32}).

Denote
\begin{equation}\label{eq42}
\Delta_*(\varepsilon)=  \frac{\varepsilon}{4g_2R_*(\varepsilon)} .
\end{equation}

Then (\ref{eq32}) and (\ref{eq42}) imply that for every finite $\Delta$-partition $\Gamma= \{ E_1,E_2,\ldots, E_N \}$  of the compact set $E$ such that $\Delta \leq \Delta_*(\varepsilon),$ the inclusion
\begin{equation}\label{eq43}
\mathbb{Z}_{\rho}^{\alpha_*(\varepsilon),Lip, R_*(\varepsilon)} \subset \mathbb{Z}_{\rho}^{\alpha_*(\varepsilon),\Gamma} + \frac{\varepsilon}{4} \mathcal{B}_2(1)
\end{equation} is satisfied where $\mathcal{B}_2(1)$ is defined by (\ref{be2}). On behalf of the inequality (\ref{eq30}) and inclusion (\ref{eq43}) we have that for every finite $\Delta$-partition $\Gamma=\left\{E_1,E_2,\ldots, E_N\right\}$  of the compact set $E$ such that $\Delta \leq \Delta_*(\varepsilon),$ the inclusion
\begin{equation}\label{eq44}
\mathbb{Z}_{\rho} \subset \mathbb{Z}_{\rho}^{\alpha_*(\varepsilon),\Gamma} + \left(c_*+\frac{1}{2}\right)\varepsilon \cdot \mathcal{B}_2(1)
\end{equation} is satisfied. Since $\mathbb{Z}_{\rho}^{\alpha_*(\varepsilon),\Gamma} \subset \mathbb{Z}_{\rho}$, then from (\ref{eq44}) we conclude that for every finite $\Delta$-partition $\Gamma=\left\{E_1,E_2,\ldots, E_N\right\}$  of the compact set $E$ such that $\Delta \leq \Delta_*(\varepsilon),$ the inequality
\begin{equation}\label{eq45}
h_2 \left(\mathbb{Z}_{\rho}, \mathbb{Z}_{\rho}^{\alpha_*(\varepsilon),\Gamma}\right) \leq \left(c_*+\frac{1}{2}\right)\varepsilon
\end{equation} is verified.

For given $\alpha_*(\varepsilon),$ finite $\Delta$-partition $\Gamma=\left\{E_1,E_2,\ldots, E_N\right\}$ of the compact set $E$ and uniform partition $\Lambda=\left\{0=r_0,r_1,\ldots, r_q=\alpha_*(\varepsilon)\right\}$ of the closed interval $[0,\alpha_*(\varepsilon)]$ where $\delta =r_{j+1}-r_j,$ $j=0,1,\ldots, q-1,$ is the diameter of the uniform partition $\Lambda,$ we define new set of control functions, setting
\begin{equation}\label{eq46}
V_{\rho}^{\alpha_*(\varepsilon), \Gamma, \Lambda}= \big\{ u(\cdot)\in V_{\rho}^{\alpha_*(\varepsilon),\Gamma} : u(s)=u_i \ \mbox{for every} \ s \in E_i \  \mbox{and} \ \left\|u_i\right\| \in \Lambda,  \ i=1,2,\ldots, N  \big\}.
\end{equation} By symbol $\mathbb{Z}_{\rho}^{\alpha_*(\varepsilon),\Gamma,\Lambda}$ we denote the set of trajectories of the system (\ref{ue1}) generated by all control functions $u(\cdot)\in V_{\rho}^{\alpha_*(\varepsilon),\Gamma,\Lambda}.$

Let us prove that for every finite $\Delta$-partition $\Gamma=\left\{E_1,E_2,\ldots, E_N\right\}$ of the compact set $E$ and uniform partition $\Lambda=\left\{0=r_0,r_1,\ldots, r_q=\alpha_*(\varepsilon)\right\}$ of the closed interval $[0,\alpha_*(\varepsilon)],$ the inequality
\begin{equation}\label{eq47}
h_2 \left(\mathbb{Z}_{\rho}^{\alpha_*(\varepsilon),\Gamma}, \mathbb{Z}_{\rho}^{\alpha_*(\varepsilon),\Gamma,\Lambda}\right) \leq g_2 \delta
\end{equation} is held where $g_2$ is defined by (\ref{g2}), $\delta =r_{j+1}-r_j,$ $j=0,1,\ldots, q-1,$ is the diameter of the uniform partition $\Lambda$.

Let us choose an arbitrary trajectory $\tilde{y}(\cdot)\in \mathbb{Z}_{\rho}^{\alpha_*(\varepsilon),\Gamma}$ generated by the control function $\tilde{w}(\cdot)\in V_{\rho}^{\alpha_*(\varepsilon),\Gamma}.$ On behalf of (\ref{eq31}) we have
\begin{eqnarray}\label{eq48}
\left\{
\begin{array}{lll}
\left\| \tilde{w}(\cdot)\right\|_2 \leq \rho, \ \ \left\|\tilde{w}(s)\right\| \leq \alpha_*(\varepsilon), \ \mbox{for every} \ s\in E,  \\
\tilde{w}(s) = \tilde{w}_i  \ \mbox{for every} \ s \in E_i, \ i=1,2,\ldots, N.
\end{array}
\right.
\end{eqnarray}

From (\ref{eq48}) it follows that if $\tilde{w}_i <\alpha_*(\varepsilon),$ then there exists $j_i=0,1,\ldots, q-1$ such that
\begin{equation}\label{eq49}
\| \tilde{w}_i \| \in \left[r_{j_i}, r_{j_i +1}\right) \, .
\end{equation}

Define new control function $\tilde{w}_*(\cdot):E\rightarrow \mathbb{R}^m$, setting
\begin{eqnarray}\label{eq50}
\tilde{w}_*(s)=\left\{
\begin{array}{lll}
\displaystyle \frac{\tilde{w}_i}{\left\| \tilde{w}_i\right\|} r_{j_i}  & \mbox{if}  & 0< \left\|\tilde{w}_i\right\| < \alpha_*(\varepsilon), \\
\displaystyle \ \ \tilde{w}_i  & \mbox{if} & \left\|\tilde{w}_i\right\|=0 \ \mbox{or} \  \left\|\tilde{w}_i\right\|=\alpha_*(\varepsilon)
\end{array}
\right.
\end{eqnarray} where $s \in E_i,$ $i=1,2,\ldots, N,$ $r_{j_i}$ is defined in (\ref{eq49}). It is not difficult to verify that $\tilde{w}_*(\cdot)\in V_{\rho}^{\alpha_*(\varepsilon),\Gamma,\Lambda}.$  (\ref{eq48}), (\ref{eq49}), (\ref{eq50}) and the equality $r_{j_i+1}-r_{j_i}=\delta$ imply that
\begin{equation}\label{eq51}
\left\|\tilde{w}(s)- \tilde{w}_*(s)\right\| \leq \delta
\end{equation} for every $s \in E.$ Let $\tilde{y}_*(\cdot): E \rightarrow \mathbb{R}^n$ be the trajectory of the system (\ref{ue1}) generated by the control function $\tilde{w}_*(\cdot)\in V_{\rho}^{\alpha_*(\varepsilon),\Gamma,\Lambda}.$ Then we have that  $\tilde{y}_*(\cdot)\in \mathbb{Z}_{\rho}^{\alpha_*(\varepsilon),\Gamma,\Lambda}.$ From Conditions 2.A, 2.B, 2.C, 2.D, Proposition \ref{prop2.2}, (\ref{eq51}) and Cauchy-Schwarz inequality it follows that
\begin{align}\label{eq52}
\left\| \tilde{y}(\xi)-\tilde{y}_*(\xi) \right\| & \leq \frac{\lambda}{1-\kappa_0} \int_{E} \big[\gamma_1(\xi,s)+ \kappa_2\left\|\tilde{w}(s)\right\|\big]\cdot \left\| \tilde{y}(s)-\tilde{y}_*(s) \right\| ds \nonumber \\ & \quad + \frac{\lambda \kappa_2 \beta_* \left[\mu(E)\right]^{\frac{1}{2}} \delta}{1-\kappa_0} +\frac{\lambda \left[\mu(E)\right]^{\frac{1}{2}} \delta}{1-\kappa_0} \left( \int_E \left\|K_2(\xi,s,0)\right\|^2 ds \right)^{\frac{1}{2}}
\end{align} for a.a. $\xi \in E.$ Since $\tilde{w}(\cdot)\in V_{\rho}^{\alpha_*(\varepsilon),\Gamma} \subset V_{\rho},$ then by virtue of (\ref{g2}), (\ref{eq52}),  Propositions \ref{prop2.4}, \ref{prop2.41} and \ref{prop2.42} we have
\begin{equation}\label{eq57}
\left\| \tilde{y}(\cdot)-\tilde{y}_*(\cdot) \right\|  \leq  \left[\displaystyle \frac{4\lambda^2 \left[\kappa_2^2 \beta_*^2 + \omega_*^2\right]\mu(E)}{\displaystyle (1-\kappa_0)^2-4\lambda^2 \left[\kappa_1^2+\rho^2 \kappa_2^2 \mu(E)\right]}\right]^{\frac{1}{2}} \cdot \delta =g_2 \delta.
\end{equation}

Since $\tilde{y}(\cdot)\in \mathbb{Z}_{\rho}^{\alpha_*(\varepsilon),\Gamma}$ is arbitrarily chosen trajectory and $\tilde{y}_*(\cdot)\in \mathbb{Z}_{\rho}^{\alpha_*(\varepsilon),\Gamma, \Lambda},$ then from (\ref{eq57}) we obtain the validity of the inclusion
\begin{equation}\label{eq57a}
\mathbb{Z}_{\rho}^{\alpha_*(\varepsilon),\Gamma}\subset  \mathbb{Z}_{\rho}^{\alpha_*(\varepsilon),\Gamma,\Lambda} + g_2 \delta \cdot \mathcal{B}_2(1)
\end{equation} where $g_2$ is defined by \eqref{g2}, $\mathcal{B}_2(1)$ is defined by \eqref{be2}, $\delta$ is the diameter of the uniform partition $\Lambda$. Since $\mathbb{Z}_{\rho}^{\alpha_*(\varepsilon),\Gamma,\Lambda} \subset \mathbb{Z}_{\rho}^{\alpha_*(\varepsilon),\Gamma},$ then (\ref{eq57a}) completes the proof of the valdity of the inequality (\ref{eq47}).

Let us set
\begin{equation}\label{eq58}
\delta_*(\varepsilon)=  \frac{\varepsilon}{4g_2} \, .
\end{equation}

Then (\ref{eq47}) and (\ref{eq58}) imply that for every finite $\Delta$-partition $\Gamma= \{E_1,E_2,\ldots, E_N\}$  of the compact set $E$ and for every uniform partition $\Lambda=\{0=r_0,r_1,\ldots, r_q=\alpha_*(\varepsilon)\}$  of the closed interval $[0,\alpha_*(\varepsilon)]$ such that $\delta \leq \delta_*(\varepsilon),$ the inequality
\begin{equation}\label{eq59}
h_2 \left(\mathbb{Z}_{\rho}^{\alpha_*(\varepsilon),\Gamma}, \mathbb{Z}_{\rho}^{\alpha_*(\varepsilon),\Gamma,\Lambda}\right) \leq \frac{\varepsilon}{4}
\end{equation} is satisfied where $\delta$ is the diameter of the uniform partition $\Lambda$. By virtue of the inequalities (\ref{eq45}) and  (\ref{eq59}) we have that for every finite $\Delta$-partition $\Gamma=\left\{E_1,E_2,\ldots, E_N\right\}$  of the compact set $E$, for every uniform partition $\Lambda=\left\{0=r_0,r_1,\ldots, r_q=\alpha_*(\varepsilon)\right\}$  of the closed interval $[0,\alpha_*(\varepsilon)]$ such that $\Delta \leq \Delta_*(\varepsilon),$  $\delta \leq \delta_*(\varepsilon),$ the inequality
\begin{equation}\label{eq60}
h_2 \left(\mathbb{Z}_{\rho}, \mathbb{Z}_{\rho}^{\alpha_*(\varepsilon),\Gamma}\right) \leq \left(c_*+\frac{3}{4}\right)\varepsilon
\end{equation} is verified where $c_*$ is defined by (\ref{ce*}), $\delta$ is the diameter of the uniform partition $\Lambda$.

Finally, us prove that for every finite $\Delta$-partition $\Gamma=\left\{E_1,E_2,\ldots, E_N\right\}$ of the compact set $E$ and uniform partition $\Lambda=\left\{0=r_0,r_1,\ldots, r_q=\alpha_*(\varepsilon)\right\}$ of the closed interval $[0,\alpha_*(\varepsilon)]$ and $\sigma >0$ the inequality
\begin{equation}\label{eq62}
h_2 \left(\mathbb{Z}_{\rho}^{\alpha_*(\varepsilon),\Gamma,\Lambda}, \mathbb{Z}_{\rho}^{\alpha_*(\varepsilon),\Gamma,\Lambda,\sigma}\right) \leq g_2 \alpha_*(\varepsilon)\sigma
\end{equation} holds where $g_2$ is defined by (\ref{g2}), $\mathbb{Z}_{\rho}^{\alpha_*(\varepsilon),\Gamma,\Lambda}$ and $\mathbb{Z}_{\rho}^{\alpha_*(\varepsilon),\Gamma,\Lambda,\sigma}$
are the sets of trajectories of the system \eqref{ue1} generated by the set of control functions $V_{\rho}^{\alpha_*(\varepsilon),\Gamma,\Lambda}$ and $V_{\rho}^{\alpha_*(\varepsilon),\Gamma,\Lambda,\sigma}$ respectively. The set $V_{\rho}^{\alpha_*(\varepsilon),\Gamma,\Lambda}$ is defined by \eqref{eq46}, $V_{\rho}^{\alpha_*(\varepsilon),\Gamma,\Lambda,\sigma}$ is defined by \eqref{fincon}.

Let us choose an arbitrary trajectory $\tilde{z}(\cdot)\in \mathbb{Z}_{\rho}^{\alpha_*(\varepsilon),\Gamma,\Lambda}$ generated by the control function $\tilde{v}(\cdot)\in V_{\rho}^{\alpha_*(\varepsilon),\Gamma,\Lambda}.$ On behalf of (\ref{eq46}) we have
\begin{eqnarray}\label{eq63}
\left\{
\begin{array}{lll}
\| \tilde{v}(\cdot)\|_2 \leq \rho, \ \| \tilde{v}(s)\| \leq \alpha_*(\varepsilon) \ \mbox{for every} \ s\in E, \\
\| \tilde{v}(s) \| = r_{j_i} \in \Lambda  \ \mbox{for every} \ s \in E_i, \ i=1,2,\ldots, N, \ \
\sum_{i=1}^{N}\mu(E_i)r_{j_i}^2 \leq \rho^2.
\end{array}
\right.
\end{eqnarray}

(\ref{eq63}) implies that for each $i=1,2,\ldots, N$ there exists $g_i\in S=\{x\in \mathbb{R}^m: \left\|x\right\|=1\}$ such that
\begin{equation}\label{eq64}
\tilde{v}(s)  = r_{j_i}\cdot g_i
\end{equation}
for every $s \in E_i.$ Since $g_i\in S$, $S_{\sigma}$ is a finite $\sigma$-net on $S,$ then for each $g_i\in S$ there exists $b_{l_i}\in S_{\sigma}=\{ b_1,b_2,\ldots, b_c\}$ such that $\left\|g_i-b_{l_i}\right\| \leq \sigma.$ Define new control function $\tilde{v}_*(\cdot):E\rightarrow \mathbb{R}^m$, setting
\begin{equation}\label{eq65}
\tilde{v}_*(s)  = r_{j_i}\cdot b_{l_i}
\end{equation}
for every $s \in E_i,$  $i=1,2,\ldots, N.$ (\ref{eq63}), (\ref{eq64}) and (\ref{eq65}) yield that $\tilde{v}_*(\cdot)\in V_{\rho}^{\alpha_*(\varepsilon),\Gamma,\Lambda,\sigma}$ and
\begin{equation}\label{eq66}
\left\| \tilde{v}(s)- \tilde{v}_*(s)\right\|  \leq   \alpha_*(\varepsilon) \cdot \sigma
\end{equation} for every $s\in E.$ Let $\tilde{z}_*(\cdot):E\rightarrow \mathbb{R}^n$ be the trajectory of the system (\ref{ue1}) generated by the control function $\tilde{v}_*(\cdot).$ Then $\tilde{z}_*(\cdot)\in \mathbb{Z}_{\rho}^{\alpha_*(\varepsilon),\Gamma,\Lambda,\sigma}.$ Conditions 2.A, 2.B, 2.C, 2.D, \eqref{eq66} and Proposition \ref{prop2.2} imply
\begin{align}\label{eq67}
& \left\| \tilde{z}(\xi)-\tilde{z}_*(\xi) \right\| \leq  \frac{\lambda}{1-\kappa_0} \int_{E} \big[\gamma_1(\xi,s)+ \kappa_2\left\|\tilde{v}(s)\right\|\big]\cdot \left\| \tilde{z}(s)-\tilde{z}_*(s) \right\| ds \nonumber \\ & \qquad + \frac{\lambda \kappa_2 \beta_* \left[\mu(E)\right]^{\frac{1}{2}} \alpha_*(\varepsilon)\sigma}{1-\kappa_0} +\frac{\lambda \left[\mu(E)\right]^{\frac{1}{2}} \alpha_*(\varepsilon)\sigma}{1-\kappa_0} \left( \int_E \left\|K_2(\xi,s,0)\right\|^2 ds \right)^{\frac{1}{2}}
\end{align} for a.a. $\xi \in E.$ Since $\tilde{v}(\cdot)\in V_{\rho}^{\alpha_*(\varepsilon),\Gamma,\Lambda} \subset V_{\rho},$ then from (\ref{g2}), (\ref{eq67}), Propositions \ref{prop2.4}, \ref{prop2.41} and \ref{prop2.42} we obtain that
\begin{equation}\label{eq72}
\left\| \tilde{z}(\cdot)-\tilde{z}_*(\cdot) \right\|_2  \leq  \left[\displaystyle \frac{4\lambda^2 \left[\kappa_2^2 \beta_*^2 + \omega_*^2\right]\mu(E)}{\displaystyle (1-\kappa_0)^2-4\lambda^2 \left[\kappa_1^2+\rho^2 \kappa_2^2 \mu(E)\right]}\right]^{\frac{1}{2}} \cdot \alpha_*(\varepsilon) \sigma =g_2 \alpha_*(\varepsilon) \sigma.
\end{equation}

Since $\tilde{z}(\cdot)\in \mathbb{Z}_{\rho}^{\alpha_*(\varepsilon),\Gamma,\Lambda}$ is arbitrarily chosen trajectory, $\tilde{z}_*(\cdot)\in \mathbb{Z}_{\rho}^{\alpha_*(\varepsilon),\Gamma, \Lambda,\sigma},$ then from (\ref{eq72}) we obtain the validity of the inclusion
\begin{equation}\label{eq73}
\mathbb{Z}_{\rho}^{\alpha_*(\varepsilon),\Gamma, \Lambda}\subset  \mathbb{Z}_{\rho}^{\alpha_*(\varepsilon),\Gamma,\Lambda,\sigma} + g_2\alpha_*(\varepsilon) \sigma \mathcal{B}_2(1)
\end{equation} where $\mathcal{B}_2(1)$ is defined by (\ref{be2}). Since $\mathbb{Z}_{\rho}^{\alpha_*(\varepsilon),\Gamma,\Lambda,\sigma} \subset \mathbb{Z}_{\rho}^{\alpha_*(\varepsilon),\Gamma,\Lambda},$ then (\ref{eq73}) completes the proof of the validity of the inequality (\ref{eq62}).

Let us set
\begin{equation}\label{eq74}
\sigma_*(\varepsilon,\alpha_*(\varepsilon))=  \frac{\varepsilon}{4g_2 \alpha_*(\varepsilon)}
\end{equation} where $g_2$ is defined by (\ref{g2}). (\ref{eq62}) and (\ref{eq74}) imply that for every finite $\Delta$-partition $\Gamma=\left\{E_1,E_2,\ldots, E_N\right\}$  of the compact set $E$, for every uniform partition $\Lambda=\left\{0=r_0,r_1,\ldots, r_q=\alpha_*(\varepsilon)\right\}$  of the closed interval $[0,\alpha_*(\varepsilon)]$ and for every finite $\sigma$-net $S_{\sigma}$ such that $\sigma \leq \sigma_*(\varepsilon,\alpha_*(\varepsilon)),$ the inequality
\begin{equation}\label{eq75}
h_2 \left(\mathbb{Z}_{\rho}^{\alpha_*(\varepsilon),\Gamma,\Lambda}, \mathbb{Z}_{\rho}^{\alpha_*(\varepsilon),\Gamma,\Lambda,\sigma}\right) \leq \frac{\varepsilon}{4}
\end{equation} is satisfied.

By virtue of the inequalities (\ref{eq60}) and  (\ref{eq75}) we conclude that for every finite $\Delta$-partition $\Gamma=\left\{E_1,E_2,\ldots, E_N\right\}$  of the compact set $E$, for every uniform partition $\Lambda=\left\{0=r_0,r_1,\ldots, r_q=\alpha_*(\varepsilon)\right\}$  of the closed interval $[0,\alpha_*(\varepsilon)]$ and for every finite $\sigma$-net $S_{\sigma}$ such that $\Delta \leq \Delta_*(\varepsilon),$  $\delta \leq \delta_*(\varepsilon),$ $\sigma \leq \sigma_*(\varepsilon,\alpha_*(\varepsilon)),$ the inequality \eqref{os1} is verified where  $\alpha_*(\varepsilon)>0$, $\Delta_*(\varepsilon)>0,$ $\delta_*(\varepsilon)>0$ and $\sigma_*(\varepsilon,\alpha_*(\varepsilon))>0$ are defined by (\ref{alfa*}), (\ref{eq42}), (\ref{eq58}) and (\ref{eq74}) respectively, $\delta=r_{j+1}-r_j$, $j=0,1,\ldots, q-1$, is the diameter of the uniform partition $\Lambda$.
\end{proof}

\section{Conclusions}

Using an algorithm presented in \cite{gus4} for specifying and aligning of the elements of a finite $\sigma$-net $S_{\sigma}=\left\{b_1,b_2, \ldots, b_c\right\}$ and an algorithm for aligning of the elements of the uniform partition $\Lambda=\{0=r_0,r_1,\ldots, r_q=\alpha_*(\varepsilon)\}$ satisfying the inequality $\sum_{i=0}^{N-1} \mu (E_i) r_{j_i}^2 \leq r^2$ and applying a numerical method for calculation of the trajectory of the system (\ref{ue1}) generated by the appropriate piecewise-constant control function, it is possible to construct the set of trajectories of the system (\ref{ue1}). Note that the number $\left(\sum_{i=0}^{N-1} \mu (E_i) r_{j_i}^2\right)^{\frac{1}{2}}$ characterizes the consumed total control resource when the piecewise-constant control function $u(s)=r_{j_i} b_{l_i}$, $s \in E_i$, $r_{j_i} \in \Lambda$, $b_{l_i}\in S_{\sigma}$, $i=1,2,\ldots, N$, is chosen as a control effort.

From inclusion $\mathbb{Z}_{\rho}^{\alpha_*(\varepsilon),\Gamma,\Lambda,\sigma} \subset \mathbb{Z}_{\rho}$ it follows that the presented approximation is an internal approximation. One of the advantages of the presented approximation method is that every function from approximating set is also the trajectory of the system, and therefore, the method can be used for specifying of the trajectories with prescribed properties.


\begin{thebibliography}{99}


\bibitem{buz}
M.E.~Buzikov and A.A.~Galyaev, Time-optimal interception of a moving target by a Dubins car, \emph{Autom. Remote Contr.} \textbf{82}(5) (2021),
745-758.

\bibitem{cher}
F.L.~Chernousko and A.I.~Ovseevich, Some properties of optimal ellipsoids that approximate attainable sets, \emph{Dokl. Akad. Nauk} \textbf{388}(4) (2003), 462-465.

\bibitem{ers}
A.A.~Ershov, A.V.~Ushakov and V.~N.~Ushakov, An approach problem for a control system with a compactum in the phase space in the presence of phase constraints, \emph{Sb. Mat.} \textbf{210}(8) (2019), 1092-1128.

\bibitem{kur}
A.B.~Kurzhanskii and P.~Varaiya, {\em Dynamics and Control of Trajectory Tubes. Theory and Computation}, Birkh\"{a}user, Cham, 2014.

\bibitem{pat}
V.S.~Patsko and A.A.~Fedotov, The structure of the reachable set for a Dubins car with a strictly one-sided turn, \emph{Tr. Inst. Mat. Mekh. UrO RAN}, \textbf{25}(3) (2019),  171-187.

\bibitem{aub}
J.-P.~Aubin and A.~Cellina, {\em Differential Inclusions. Set-Valued Maps and Viability Theory},  Springer-Verlag, Berlin, 1984.

\bibitem{bre}
A.~Bressan and G.~Facchi, Trajectories of differential inclusions with state constraints, \emph{J. Differ. Equ.} \textbf{250} (4)(2011), 2267-2281.

\bibitem{bla}
V.I.~Blagodatskikh and A.F.~Filippov, Differential inclusions and optimal control, \emph{Trudy Mat. Inst. Steklov}. \textbf{169} (1985), 194-252.

\bibitem{cla}
F.H.~Clarke, Yu.S.~Ledyaev, R.J.~Stern and P.R.~Wolenski, {\em Nonsmooth Analysis and Control Theory}, Springer-Verlag, New York, 1998.

\bibitem{dei}
K.~Deimling, {\em Multivalued Differential Equations}, D. Gruyter, Berlin, 1992.

\bibitem{pan}
A.I.~Panasyuk and V.I.~Panasyuk, An equation generated by a differential inclusion,  \emph{Mat. Zametki}, \textbf{27}(3) (1980), 429-437.

\bibitem{con}
R.~Conti, {\em Problemi di Controllo e di Controllo Ottimale}, UTET, Torino, 1974.

\bibitem{kra2}
N.N.~Krasovskii, {\em Theory of Control of Motion: Linear Systems}, Nauka, Moscow, 1968.

\bibitem{sub}
N.N.~Subbotina and A.I.~Subbotin, Alternative for the encounter-evasion differential game with constraints on the momenta of the players controls, \emph{Prikl. Mat. Meh.}  \textbf{39}(3) 1975, 397-406.

\bibitem{gus3}
Kh.G.~Guseinov, Approximation of the attainable sets of the nonlinear control systems with integral constraint on controls, \emph{Nonlinear Anal. TMA},  \textbf{71}(1-2) (2009), 622-645.

\bibitem{gus4}
Kh.G.~Guseinov and A.S.~Nazlipinar, An algorithm for approximate calculation of the attainable sets of the nonlinear control systems with integral constraint on controls, \emph{Comput. Math. Appl.} \textbf{62}(4) (2011), 1887-1895.

\bibitem{gusm}
M.I.~Gusev and I.V.~Zykov, On extremal properties of the boundary points of reachable sets for control systems with integral constraints, \emph{Tr. Inst. Mat. Mekh. UrO RAN},  \textbf{23}(1) 2017,  103-115.


\bibitem{polyak}
B.T.~Polyak, Convexity of the reachable set of nonlinear systems under $L_2$ bounded controls,  \emph{Dyn. Contin. Discrete Impuls. Syst. Ser. A Math. Anal.} \textbf{11}(2-3) (2004), 255-267.

\bibitem{rou}
P.~Rousse, P.-L.~Garoche and D.~Henrion, Parabolic set simulation for reachability analysis of linear time-invariant systems with integral quadratic constraint,  \emph{European J. Contr.}  \textbf{28} (2021), 152-167.

\bibitem{ban}
J.~Banas and A.~Chlebowicz, On integrable solutions of a nonlinear Volterra integral equation under Caratheodory
conditions, \emph{Bull. Lond. Math. Soc.} \textbf{41} (2009), 1073-1084.

\bibitem{bog}
V.I.~Bogachev, A.I.~Kirillov and S.V.~Shaposhnikov, On probability and integrable solutions to the stationary Kolmogorov equation, \emph{Dokl. Math.} \textbf{83}(3) (2011),   309-313.

\bibitem{cor}
C.~Corduneanu, {\em Integral Equations and Applications}, Cambridge University Press, Cambridge, 1991.

\bibitem{kan}
L.V. Kantorovich and G.P. Akilov, {\em Functional Analysis},  Nauka, Moscow, 1977.

\bibitem{krasn}
M.A.~Krasnoselskii, P.P.~Zabreiko, E.I.~Pustylnik and P.E.~Sobolevskii, {\em Integral Operators in Spaces of Summable Functions},   Noordhoff International Publishing, Leyden, 1976.

\bibitem{ang}
T.S.~Angell, R.K.~George and J.P.~Sharma, Controllability of Urysohn integral inclusions of Volterra type, \emph{Electron. J. Diff. Equat.}, \textbf{79} (2010), 1-12.

\bibitem{bal}
E.J.~Balder, On existence problems for the optimal control of certain nonlinear integral equations of Urysohn type, \emph{J. Optim. Theory Appl.} \textbf{42} (1984),  447-465.

\bibitem{car}
D.A.~Carlson, An elementary proof of the maximum principle for optimal control problems governed by a Volterra integral equation, \emph{J. Optim. Theory Appl}. \textbf{54} (1987), 43-61.

\bibitem{hus2}
A.~Huseyin, N.~Huseyin and Kh.G.~Guseinov, Approximation of the integral funnel of a nonlinear control system with limited control resources,  Minimax Theory Appl. \textbf{5}(2) (2020),  327-346.

\bibitem{hus3}
N.~Huseyin, A.~Huseyin and Kh.G.~Guseinov, Approxmation of the set of trajectories of the nonlinear control system with limited control resources, \emph{Math. Model. Anal.} \textbf{23}(1) (2018),  152-166.

\bibitem{hus7}
N.~Huseyin and A.~Huseyin, On the compactness of the set of $L_2$ trajectories of the control system, \emph{Nonlin. Anal. Model. Control} \textbf{23}(3) (2018), 423-436.

\bibitem{whe}
R.L.~Wheeden and A.~Zygmund, {\em Measure and Integral. An Introduction to Real Analysis}, M. Dekker, New York, 1977.

\end{thebibliography}
\end{document}